\documentclass[11pt, a4paper]{amsart}
\usepackage{upgreek}

\usepackage{amsfonts, amsthm, amssymb, amsmath, stmaryrd}
\usepackage{mathrsfs,array}
\usepackage{eucal,fullpage,times,color,enumerate,accents}
\usepackage{url}
\usepackage[euler-digits]{eulervm}

\newtheorem{innercustomthm}{Theorem}
\newenvironment{customthm}[1]
  {\renewcommand\theinnercustomthm{#1}\innercustomthm}
  {\endinnercustomthm}

\usepackage[backref]{hyperref}
\hypersetup{
  colorlinks   = true,          
  urlcolor     = blue,          
  linkcolor    = purple,          
  citecolor   = blue             
}

\DeclareMathAlphabet{\mathbbold}{U}{bbold}{m}{n}

\usepackage{color}
\usepackage{mathrsfs}
\usepackage{amssymb}
\usepackage{tikz}
\usepackage{tikz-cd}
\usepackage{xypic}
\usepackage{bm}
\usepackage{enumerate}
\usetikzlibrary{calc}
\usetikzlibrary{fadings}

\title{{\larger\larger H}ahn analytification and connectivity\\ of higher rank tropical varieties}
\author{{\larger T}{\smaller yler}\ \ {\larger F}{\smaller oster}\ \ \&\ \ {\larger D}{\smaller hruv}\ \ {\larger R}{\smaller anganathan}}
\date{\today}

\address{{\bf Tyler Foster}\newline Department of Mathematics, University of Michigan}
\email{tyfoster@umich.edu}

\address{{\bf Dhruv Ranganathan}\newline Department of Mathematics, Yale University}
\email{dhruv.ranganathan@yale.edu}

\newtheorem{theorem}{Theorem}[subsection]
\newtheorem{corollary}[theorem]{Corollary}

\newtheorem{proposition}[theorem]{Proposition}
\newtheorem{definition}[theorem]{Definition}
\newtheorem{warning}[theorem]{Warning}

\newtheorem{construction}[theorem]{Construction}
\newtheorem{example}[theorem]{Example}
\newtheorem{quasi-theorem}[theorem]{Quasi-Theorem}
\newtheorem{blank remark}[theorem]{}

\newtheorem{rem1}[theorem]{Remark}
\newenvironment{remark}{\begin{rem1}\em}{\end{rem1}}

\newtheorem{not1}[theorem]{Notation}

\newcommand{\PP}{\mathbb{P}}         
\newcommand{\QQ} {{\mathbb Q}}		
\newcommand{\RR} {{\mathbb R}}		
\newcommand{\ZZ} {{\mathbb Z}}

\newcommand{\Hom}{\operatorname{Hom}}

\DeclareMathOperator{\val}{val}

\DeclareMathOperator{\spec}{Spec}

\DeclareMathOperator{\ev}{{ev}}
\DeclareMathOperator{\red}{{red}}

\newcommand{\cal}{\mathcal}

\def\cM{{\cal M}}

\def\fH{\mathfrak{H}}

\def\fp{\mathfrak{p}}

\def\trop{\mathrm{trop}}

\newcommand{\Mbar}{\overline{\cM}}

\def\lra{\longrightarrow}

\newcommand{\mono}{\!\xymatrix{{}\ar@{^{(}->}[r]&{}}\!}

\newif\ifshow
\showfalse

\begin{document}
\pagestyle{plain}
\maketitle

\begin{abstract}
We show that the tropicalization of a connected variety over a higher rank valued field is a path connected topological space. This establishes an affirmative answer to a question posed by Banerjee~\cite{Ban13}. Higher rank tropical varieties are studied as the images of "Hahn analytifications", introduced in this paper. A Hahn analytification is a space of valuations on a scheme over a higher rank valued field. We prove that the Hahn analytification is related to higher rank tropicalization by means of an inverse limit theorem, extending well-known results in the non-Archimedean case. We also establish comparison results between the Hahn analytification and the Huber and Berkovich analytifications, as well as the Hrushovski-Loeser stable completion. 
\end{abstract}

\setcounter{section}{0}
\setcounter{subsection}{0}

\section{Introduction}

	In recent years, numerous authors have studied the relationship between tropical and analytic geometry over rank-$1$ valued fields \cite{BPR, EKL06, Gub13, Pay09,R12}. Two of the fundamental results in the subject are that tropicalizations of subvarieties of tori have a polyhedral structure, proved by Bieri and Groves~\cite{BG84}, and that the tropicalization of a connected variety is connected, proved by Einsielder, Kapranov, and Lind~\cite{EKL06}. The latter may also be seen as an elementary consequence of a deep theorem in Berkovich's theory of analytic spaces~\cite[Chapter 3]{Ber90}. 
	
	The purpose of the present text is to extend the study of this relationship between tropical and analytic geometry to valued fields of arbitrary finite rank, and extend these two fundamental results. We introduce a new form of tropicalization over higher rank valued fields and prove that the tropicalization of a connected subvariety of a torus gives rise to a path connected topological space. This resolves a question posed by Banerjee in~\cite{Ban13}. The proof of connectivity relies on a new theory of analytification over higher rank valued fields, which reduces to Berkovich analytification in rank-$1$. We describe in detail the relationship between this theory and the Berkovich and Huber approaches to geometry over non-Archimedean fields~\cite{Ber90,Hub}, as well as Hrushovski and Loeser's theory of stable completions \cite{HL10}. See Section~\ref{sec: main} for a precise statement of the main results. 

	This paper is complemented by~\cite{FR15b}, in which we study multistage degenerations of toric varieties over higher rank valuation rings. While valued fields and their geometry are of basic interest, our central motivation is toward the theory of limit linear series. Multistage degenerations can be used to interpolate between the compact type theory of Eisenbud and Harris~\cite{EH86} and the maximally degenerate tropical theory used by Jensen and Payne in~\cite{JP15A}~\cite{JP15B}. As pointed out in~\cite[Remark 1.4]{JP15A}, the connection between these approaches remains an important open question. In a later work, we intend to study this connection and use it to establish further results in the vein of ~\cite{EH86}, ~\cite{JP15A}, and ~\cite{JP15B}.

\subsection{Hahn analytification}\label{sec: hahn-preliminary} Let $K$ be a field equipped with a valuation
$$
\nu: K^\times\lra \Gamma,
$$
where $\Gamma$ is a totally ordered abelian group. Fix a positive integer $k$ and denote by $\RR^{(k)}$ the \newline $k$-fold product of $\RR$, equipped with the lexicographic ordering. Choose an order preserving homomorphism $\rho: \Gamma\lra \RR^{(k)}$. 

Let $X = \spec_{\ \!}A$ be an affine $K$-variety. The \textit{Hahn analytification of $X$} is the set of ring valuations
$$
|X^\fH|\ :=\ \left\{A\xrightarrow{\val}\RR^{(k)}\sqcup \{\infty\}: \val(z) = \rho\circ \nu(z), \ \textnormal{for all } z\in K\right\}.
$$
By definition, a {\em ring valuation} $\val:A\lra\RR^{(k)}\sqcup \{\infty\}$ is a map satisfying $\val(0)=\infty$, $\val(ab) = \val(a)+\val(b)$, and $\val(a+b)\geq \min\{\val(a),\val(b)\}$.

We put two distinct topologies on the set $|X^{\fH}|$:

\noindent
{\bf(The Extended Order Topology on $|X^\fH|$.)} Give $\RR^{(k)}\sqcup \{\infty\}$ the {\em extended order topology} by declaring that $r<\infty$ for all $r\in \RR^{(k)}$. Equip the set $|X^\fH|$ with the weakest topology making the evaluation functions
\begin{eqnarray*}
\ev_f\ :\ \ |X^\fH|&\lra& \RR^{(k)}\sqcup \{\infty\}\\
\val_x&\longmapsto& \val_x(f)
\end{eqnarray*}
continuous with respect to this extended order topology on $\RR^{(k)}\sqcup\{\infty\}$, for all $f\in A$. That is, $\infty$ is a global maximum, and $a<\infty$ for all $a\in \RR^{(k)}$. Denote the resulting topological space by $X^\fH$. We will refer to this space as the Hahn analytification in the extended order topology. 

\noindent
{\bf (The Extended Euclidean Topology on $X^\fH$.)} Extend the Euclidean topology on $\RR$ to the topology on $\RR_{\infty} = \RR\sqcup \{\infty\}$ for which the completed rays $(a,\infty]$, $a\in \RR$, form a basis of open neighborhoods at $\infty$. The {\em extended Euclidean topology} on $\RR^{(k)}\sqcup\{\infty\}$ is the subspace topology obtained by identifying $\RR^{(k)}\sqcup\{\infty\}$ with the subspace $\mathbb{R}^{k}\cup\big\{(\infty,\dots,\infty)\big\}$ of $(\RR_\infty)^k$. Equip $|X^\fH|$ with the weakest topology making the evaluation functions $\ev_f$ continuous with respect to the resulting subspace topology on $\RR^{(k)}\sqcup \{\infty\}$. Denote the resulting topological space by $X^\fH_\#$. We will refer to this space as the Hahn analytification in the extended Euclidean topology. The reader may think of the point $\infty$ in $\RR^{(k)}\sqcup \{\infty\}$ as being a ``sharp'' corner that partially compactifies $\RR^k$. 

\begin{remark}{\bf(Basic topological properties).}
When $k = 1$, the extended Euclidean topology and extended order topology on $\RR\sqcup \{\infty\}$ coincide. For $k\geq 1$, the space $\RR^{(k)}\sqcup \{\infty\}$ is Hausdorff and non-compact in both topologies. When $k\geq 2$, the order topology on $\RR^{(k)}$ is strictly finer than the Euclidean topology on $\RR^{(k)}$, while the extended order and Euclidean topologies on $\RR^{(k)}\sqcup\{\infty\}$ are incomparable. 
\end{remark}

\begin{remark}{\bf (Alternative topologies).}
Note that for $k\geq 2$, there are other ways in which one might extend the Euclidean topology on $\RR^{(k)}$ to a topology on $\RR^{(k)}\sqcup \{\infty\}$. For instance, if one gives $\infty$ an open neighborhood basis consisting of sets of the form $U_{a}:=\{\infty\}\sqcup\big\{(r_{1},\dots,r_{k})\in\RR^{k}:r_{1}>a\big\}$ for all $a\in\RR$, then the resulting topology on $\RR^{(k)}\sqcup\{\infty\}$ is strictly coarser than the extended order topology on $\RR^{(k)}\sqcup\{\infty\}$. The results of this paper that make use of the extended Euclidean topology are unaffected if we replace the extended Euclidean topology on $\RR^{(k)}\sqcup\{\infty\}$ by any other topology that is path connected.
\end{remark}

\subsection{Hahn tropicalization} Let $T$ be a split algebraic torus of dimension $d$ with character lattice $M$, and let $X = \spec_{\ \!}A$ be a closed subvariety of $T$. For each point $x\in X^\fH$, i.e., for each valuation $\val_x:A\lra \RR^{(k)}\sqcup\{\infty\}$, consider the composite
	\begin{equation}\label{equation: trop composite}
	M\lra K[M]\lra A\lra \RR^{(k)}\sqcup \{\infty\}.
	\end{equation}
For characters $\chi^u,\chi^v\in M$, the valuation axioms imply $\val_x(\chi^{u+v}) = \val_x(\chi^u)+\val_x(\chi^v)$. Thus, (\ref{equation: trop composite}) can be taken to be a homomorphism of abelian groups, which we denote
	$$
	\trop(x):M\lra\RR^{(k)}.
	$$
In this way, we obtain a {\em tropicalization map}
	\begin{equation}\label{equation: the tropicalizaiton map}
	\trop: |X^\fH|\lra \Hom_{\ZZ}(M,\RR^{(k)}).
	\end{equation}
	
	The \textit{Hahn tropicalization} of $X$, denoted $|\trop(X)|$, is the image of $X^\fH$ under (\ref{equation: the tropicalizaiton map}). When no confusion can arise, we simply refer to $|\trop(X)|$ as the {\em tropicalization} of $X$. The order and Euclidean topologies on $\RR^{(k)}$ determine two distinct topologies on the set $\Hom_{\ZZ}(M,\mathbb{R}^{(k)})$. We let $\trop(X)$ denote the Hahn tropicalization with the subspace topology for the order topology on $\RR^{(k)}$, and we let $\trop_\#(X)$ denote the Hahn tropicalization with the subspace topology for the Euclidean topology. Observe that the topological space $\trop(X)$ (resp. $\trop_\#(X)$) is the continuous image of $X^\fH$ (resp. $X^\fH_\#$) under the map $\trop$. 

\begin{remark}
When $k = 1$, $\trop_\#(X)$ and $\trop(X)$ coincide with the image of the Berkovich analytification of $X$ under the standard tropicalization map. Aroca~\cite{Aroca10} previously studied tropicalizations of hypersurfaces over higher rank valued fields by extending the theory of Newton polygons. When $K$ is a field with value group equal to $\RR^{(k)}$, her definition coincides with ours. Note that Aroca's tropicalizations do not carry a topology. The higher rank tropicalization studied by Banerjee~\cite{Ban13} is the closure of $\trop_\#(X)$ in the Euclidean topology. We discuss the relationship with Banerjee's work in more detail in Section~\ref{sec: banerjee}.
\end{remark}

\subsection{Main results}\label{sec: main} If $X$ is a connected, closed subvariety of a torus over a rank-$1$ non-Archimedean field, then the usual tropicalization of $X$ is connected. This result was first proved by Einsiedler, Kapranov, and Lind~\cite{EKL06} using rigid analytic techniques, but it can also be obtained as an elementary consequence of connectivity of the Berkovich analytification $X^{\mathrm{an}}$. Our main results in the present text are extensions of this connectivity result to the higher rank setting.

\begin{customthm}{A}\label{thm: tropical-theorem}
Let $X$ be a connected subvariety of an algebraic torus over $K$. Then $\trop_\#(X)$ is a path connected topological space.
\end{customthm}

From Theorem \ref{thm: tropical-theorem} together with the results of Section~\ref{sec: banerjee} we deduce connectivity for Banerjee's tropicalization.

\begin{customthm}{B}\label{thm: tropical-theorem'}
Let $X$ be a connected subvariety of an algebraic torus over $K$. Then $\trop(X)$ is a definable and definably path connected space.
\end{customthm}

\begin{remark}
In the case where $K$ is Henselian, Ducros proves a model theoretic connectivity result for a distinct but related tropicalization that appears in~\cite[Theorem 1.2]{Duc12}. We note however that the techniques used in loc. cit. are based on model theory, and are quite different from those that appear here.
\end{remark}
	
	The connectivity results of Theorems \ref{thm: tropical-theorem} and \ref{thm: tropical-theorem'} are consequences of basic properties of the Hahn analytification itself.
	
\begin{customthm}{C}\label{thm: Hahn analytification structure}
	If $X$ is a geometrically connected $K$-variety, then $X^{\fH}_{\#}$ is path connected. Furthermore, for any $\RR^{(k)}$-valued field $F$ extending $K$, each pair of $F$-rational points $x$ and $y$ in $X_{F}:=X\times_{K}F$ are connected by a definable path in $X^{\fH}_{\!F}$.
	
	If $K$ is algebraically closed, and if $X$ is a $K$-variety that can be realized as a closed subvariety of a toric $K$-variety, then $X^{\fH}$ is a prodefinable set.
\end{customthm}

The Euclidean connectivity, definable connectivity, and prodefinability results in Theorem \ref{thm: Hahn analytification structure} are restated and proved in Theorems~\ref{thm: xH-euclidean-path-connectivity}, ~\ref{thm: xH-galois-connectivity}, and~\ref{cor: prodefinability} respectively.

\begin{remark}
Definability and prodefinability are basic concepts in logic and the theory of $o$-minimal structures~\cite{vdD98}. The relevant ideas are reviewed in Section~\ref{sec: definability}.  One can understand the definability of $\trop(X)$ as a reflection of the fact that $\trop(X)$ is a finite union of polyhedra in $(\RR^{(k)})^d$. The definable path connectivity of $\trop(X)$ is the statement that any two points in $\trop(X)$ may be connected by a path parametrized by a (generalized) interval in the ordered abelian group $\RR^{(k)}$. Moreover, this interval is embedded in $(\RR^{(k)})^d$ as a rational $1$-dimensional polyhedral complex. Note that the naive connectivity statement for the order topology is false, as $\RR^{(k)}$ is disconnected for $k\geq 2$. 
\end{remark}
	
	In Section~\ref{sec: comparison} we discuss the relationship between the Hahn analytification and the Huber adic space, and the stable completion appearing in recent work of Hrushovski and Loeser. In Section~\ref{ref: tropicalization}, we prove that the Hahn analytification and tropicalizations are related by an inverse limit theorem (Theorem~\ref{thm: inverse-limit}) in the spirit of~\cite{FGP,Pay09}. The prodefinability of $X^\fH$ is an immediate consequence.

\subsection*{Acknowledgements} We are grateful to Sam Payne for conversations that helped shape this project, and to Dan Abramovich, Dori Bejleri, Yoav Len, and Sam Payne for comments on previous drafts. Important progress was made while the authors were attending the AMS Western Sectional Meeting on Combinatorial Algebraic Geometry and we thank Farbod Shokrieh and Madhusudan Manjunath for creating that opportunity. We have benefited from many conversations with friends and colleagues including Matt Baker, Dan Corey, Walter Gubler, and Jeremy Usatine. We thank the anonymous referees for comments that helped improve the article. T.F. was partially supported by NSF RTG grant DMS-09343832. D.R. was partially supported by NSF grant CAREER DMS-1149054. This work was completed while D.R. was visiting Brown University and the University of Michigan, and it is a pleasure to thank these institutions here. 

\section{Hahn analytification and connectivity theorems}\label{sec: hahn-analytification}

Throughout the present section, fix the following notation. Let $K$ be a field equipped with a valuation $\nu:K^\times\to \Gamma$. Fix an order preserving homomorphism $\rho:\Gamma\to \RR^{(k)}$. We refer to the data of $K$ together with the maps $\nu$ and $\rho$ as a \textit{Hahn valued field}. Let $R$ denote the valuation ring associated to the valuation $\rho\circ \nu$. A {\em Hahn field extension} of the triple $(K,\nu,\rho)$ is a field extension $L$ of $K$ equipped with a valuation $\nu_L:L^\times\to \RR^{(k)}$ such that $\nu_L(z) = \rho\circ \nu(z)$ for all $z\in K$. 

\begin{remark}
It is often desirable to choose $\rho$ to be an embedding into $\RR^{(k)}$. Such an embedding always exists by a theorem of Hahn~\cite{Hahn}. If the rank\footnote{Recall that the \textit{rank} of a totally ordered abelian group is the number of proper convex subgroups contained in it.} of $\Gamma$ is finite, $k$ can be taken to be equal to the rank of $\Gamma$. Nonetheless, it is convenient to allow $\rho$ to not be non-injective in general.
\end{remark}

Let $X$ be a separated finite type $K$-scheme.  Consider pairs consisting of a Hahn valued field extension $L$ of $K$ and a point $x\in X(L)$ . There is an equivalence relation generated by declaring that
$$
(L,x)\ \sim\ (L\!',x')
$$
whenever there is an embedding $L\hookrightarrow L\!'$, such that $\nu_{L\!'}$ restricts to $\nu_L$ on the subfield $L$, and that $x\mapsto x'$ under the induced inclusion of point sets $X(L)\hookrightarrow X(L\!')$. Set theoretically, we define the {\em Hahn analytification} $|X^\fH|$ to be the resulting collection of equivalence classes of points of $X$ over valued extensions of $K$. That is,
$$
|X^\fH| \ :=\ \big\{(L,x)\big\}\big/{}_{\mbox{{\larger $\sim$}}}\ .
$$

	Suppose $X = \spec(A)$. Given a valuation $\val:A\to \RR^{(k)}\sqcup \{\infty\}$, the kernel $\val^{-1}(\infty)$ is a prime ideal. In turn this gives rise to a Hahn valuation  fraction field $\mathrm{Frac}(A/\fp)$. Thus, the above definition of analytification is equivalent for affine schemes to the one given in Section~\ref{sec: hahn-preliminary}. As $X$ is covered by affine opens, $|X^\fH|$ is the union of the Hahn analytifications of these affine opens. As in the Berkovich setting~\cite[Section 3.4]{Ber90}, the topologies on these affines agree on their overlaps and determine a global topology on the set $|X^\fH|$. We let $X^\fH$ denote $|X^{\fH}|$ with its order topology, and we let $X^\fH_\#$ denote $|X^{\fH}|$ with its Euclidean topology.

\noindent
The construction of $X^{\fH}$ and $X^{\fH}_{\#}$ is covariantly functorial. A morphism $f:X\to Y$ of finite-type $K$-schemes induces a natural a map $f^{\fH}:|X^{\fH}|\to|Y^{\fH}|$ that is continuous in both the order and Euclidean topologies.
	
	Given a point $x\in X(K)$, the composition of evaluation at $x$ with the valuation on $K$ defines a valuation on the coordinate ring of any affine open neighborhood of $x$. In this way, $X(K)$ becomes a subset of $|X^\fH|$.
	
	If $K$ is Henselian~\cite[Chapter 4]{EP05}, then the valuation on $K$ extends uniquely to the algebraic closure, giving an inclusion of the set of closed scheme theoretic points of $X$ into $|X^\fH|$.
	
\subsection{Tower of projections and relation to Berkovich analytification}\label{sec: berk-comparison} For any $0\leq j \leq k$, there exists a unique continuous order-preserving projection $\pi^k_j:\RR^{(k)}\to \RR^{(j)}$, namely projection to the \textit{first} $j$ factors. Composing the map $\rho$ with this projection $\pi^{k}_{j}$, we obtain a new Hahn valued field with the same underlying field $K$. Correspondingly, one obtains a tower of Hahn analytifications with continuous maps between them:
	\begin{equation}\label{projective relationship}
	\begin{aligned}
	\begin{xy}
	(0,0)*+{\ X^{\fH}_{k}\ }="k";
	(18,0)*+{\ X^{\fH}_{k-1}\ }="k-1";
	(36,0)*+{\ \cdots\ }="k-2";
	(54,0)*+{\ X^{\fH}_{2}\ }="2";
	(72,0)*+{\ X^{\fH}_{1}\ }="1";
	(90,0)*+{\ X^{\fH}_{0}\ }="0";
	{\ar "k"; "k-1"};
	{\ar "k-1"; "k-2"};
	{\ar "k-2"; "2"};
	{\ar "2"; "1"};
	{\ar "1"; "0"};
	\end{xy}.
	\end{aligned}
	\end{equation}
Here $X^\fH_j$ denotes the Hahn analytification of $X$ with respect to $K$ with its valuation 
$$
K^\times \xrightarrow{\ \ \nu\ \ } \Gamma\xrightarrow{\ \ \rho\ \ } \RR^{(k)}\xrightarrow{\ \pi^k_j\ }\RR^{(j)}.
$$	

\begin{warning}\textnormal{
Some care is required in order to handle the case of $j = 0$. The set $\{0\}\sqcup \{\infty\}$ is topologized as a connected doubleton, where the open sets are $\emptyset$, $\{0\}$ and $\{0,\infty\}$, and $\infty$ is the only closed point. Observe that the projection $\RR\sqcup \{\infty\}\to\{0\}\sqcup\{\infty\}$ taking $\RR\to\{0\}$ is continuous.}
\end{warning}

We point out two special cases.

\begin{example}\label{example: Zariski spectrum}
	The space $X^\fH_0$ coincides with the set of scheme theoretic points of $X$ in the Zariski topology. To see this, observe that for $X$ affine, the map $X \to X^\fH$ taking a prime $\fp$ to the trivial valuation on $K[X]/\fp$ yields a set theoretic bijection $|X|\cong|X_0^\fH|$. To identify $X^\fH_0$ and $X$ as topological spaces, note that for any regular function $f$ on an affine scheme, $\ev_f^{-1}(\infty)$ consists of exactly those prime ideals $\fp$ that contain $f$. In other words, $\ev_f^{-1}(\infty)=V(f)$. These generate the closed sets of the Zariski topology on $X$.
\end{example}

\begin{example}
	The space $X^\fH_1$ is homeomorphic to the Berkovich analytification\footnote{Note that for $K$ a non-complete rank-$1$ field, Berkovich analytification is redefined by Hrushovski and Loeser as a space of types, recovering the usual definition when $K$ is complete. Employing this definition, we can ignore questions about whether or not $K$ is complete with the rank-$1$ valuation described. See~\cite[Section 5]{HLP} for a nice discussion.} of $X$ with respect to the composite rank-$1$ valuation 
	$$
	K^\times \xrightarrow{\ \ \nu\ \ } \Gamma \xrightarrow{\ \ \rho\ \ } \RR^{(k)}\xrightarrow{\ \pi^k_1\ }\RR.
	$$
By Example \ref{example: Zariski spectrum}, the map $X^{\fH}_{1}\to X^{\fH}_{0}$ is the continuous map, to the underlying scheme, that realizes the universal property of Berkovich analytification~\cite[Section 3.5]{Ber90}.
\end{example}

\subsection{Euclidean path connectivity}

In this subsection we prove Theorem~\ref{thm: tropical-theorem}. The result is deduced from the following result about the structure of the Hahn analytification.

\begin{theorem}\label{thm: xH-euclidean-path-connectivity}
If $X$ is a geometrically connected $K$-variety, then the topological space $X^\fH_\#$ is path connected. 
\end{theorem}

Our proof of Theorem \ref{thm: xH-euclidean-path-connectivity} requires the following auxiliary construction. 

\begin{construction}\label{valued graphs}
Let $G$ be a finite graph with edges $e_1,\ldots, e_r$. Fix an $r$-tuple
	$$
	\underline \ell\ =\ \big(\ell(e_1),\ldots, \ell(e_r)\big)\ \in\ \big(\RR^{(k)}_{\geq 0}\sqcup \{\infty\}\big)^{\!r}.
	$$
For each $i$, consider the interval
\[
\big[0,\ell(e_i)\big] = \{ \gamma\in \RR^{(k)}\sqcup\{\infty\}: 0\leq \gamma \leq \ell(e_i)\}.
\]
Let $|e_i|_\#$ be the interval $\big[0,\ell(e_i)\big]$ equipped with the subspace topology for the extended Euclidean topology on $\RR^{(k)}\sqcup\{\infty\}$. Choose a bijection between the endpoints of $\big[0,\ell(e_i)\big]$ and the vertices of $G$ incident to the edge $e_i$. We define a topological space
\[
|G(\underline \ell)|_\#\ :=\ \Big({\bigsqcup}_k |e_k|_\#\Big)\Big/\sim\ ,
\]
where the relation ``$\sim$" identifies the endpoints of $|e_i|_\#$ and $|e_j|_\#$ whenever the corresponding vertices are identified in the graph $G$. One can easily observe that the resulting topological space does not depend on the chosen orientation for the interval.
\end{construction}

	Each interval in $\RR^{(k)}\sqcup \{\infty\}$ is path connected in the extended Euclidean topology. Hence if $G$ is a connected graph, then $|G(\underline \ell)|_\#$ is path connected. 
	
	The follow construction will be used to prove our main connectivity result. Our strategy is to connect points of $X^\fH$ using Hahn analytifications of curves $C$ in $X$. Distinguished paths in $C^\fH$ will be identified by passing to an appropriate ``skeleton''.

\begin{construction}\label{remark: length assignment}
Suppose $\mathscr C$ is a proper, regular, marked semistable $R$-curve with generic fiber $C$ and special fiber $\mathscr C_0$. Assume that the fibers of $\mathscr C$ have no self intersections, and let $G$ be the marked dual graph of $\mathscr C_0$. The graph $G$ has one vertex for each irreducible component of $\mathscr C_0$ and each horizontal mark in $\mathscr{C}$, and $G$ has an edge between two vertices whenever the two corresponding components of $\mathscr{C}_{0}$ share a node, or whenever the horizontal mark corresponding to one of the vertices intersects the component of $\mathscr{C}_{0}$ corresponding to the other vertex. If $e_i$ is an edge corresponding to a node $p_i$ where two components of $\mathscr{C}_{0}$ intersect, then the local equation for $\mathscr C$ near $p_i$ is given by $x_iy_i = f_i$ for $f_i\in R$. Set $\ell(e_i) = \nu(f_i)$. If $e_i$ is an edge corresponding to the intersection of a horizontal mark with a component of $\mathscr{C}_{0}$, then set $\ell(e_i) = \infty$.
\end{construction}

As we now show, the dual graph of a model $\mathscr C$ above embeds naturally into the Hahn analytification of the generic fiber. The construction is a variant of the standard construction over rank-$1$ fields, see~\cite{BPR, MN12}. 

\begin{proposition}\label{prop: skeleton-embeds}
There is a continuous embedding $|G(\underline \ell)|_\#\hookrightarrow C^\fH_\#$. If $e$ is an edge corresponding to a marked point $p$ of the generic fiber $C$, the infinite point of $|e|_\#$ is mapped to the image of $p$ under the inclusion $C(K)\hookrightarrow C^\fH_\#$. 
\end{proposition}

\begin{proof}
Let $\omega\in |G(\underline \ell)|_\#$ be a point lying in a subspace $|e|_\#\subset |G(\underline \ell)|_\#$ that corresponds to a node $p$ between two components of the special fiber $\mathscr{C}_{0}$. We build a point of $C^\fH_\#$ as follows. Since the fibers of $\mathscr C$ have no self-intersections, Zariski locally near $p$ the $R$-model $\mathscr C$ is of the form $\mathrm{Spec}_{\ \!}R[x,y]/(xy-f)$ where $\ell(e)=\nu(f)$. For each $g = \sum a_{jk} x^jy^k$, define
	$$
	\val_\omega(g)
	\ =\ 
	\min\ \!\big\{\ \!\nu(a_{jk})+j\ \!\omega+k\ \!\big(\ell(e)-\omega\big)\ \!\big\}.
	$$
Analogously to the rank-$1$ case, $\val_\omega$ defines a valuation on $R[x,y]/(xy-f)$ that extends to the field of rational functions $K(C)$. Similarly, if $e$ corresponds to a node where a marked section of $\mathscr{C}$ intersects $\mathscr{C}_{0}$, then one may choose local coordinates to describe $\mathscr C$ as $\spec_{\ \!}R[x]$, where the marked section is cut out by $x$. Again, given $\omega\in |e|_\#$, one may construct a monomial valuation $\val_\omega$ that takes each $g = \sum a_{j} x^j$ to
\[
	\val_\omega(g)
	\ =\ 
	\min\ \!\big\{\ \!\nu(a_{j})+j\ \!\omega\ \!\big\}.
\]
The resulting map $\iota: |G(\underline{\ell})|_{\#}\to C^{\fH}_{\#}$ is a continuous inclusion. To obtain the continuous inverse on the image of $\iota$, observe that each edge $|e|_\#$ corresponding to a node (resp. a marked point) of $\mathscr C_0$ the set $|e|_{\#}$ as a subspace of $(\RR^{(k)})^{2}$ (resp. $\RR^{(k)}\sqcup\{\infty\}$). The inverse map to $\iota$ on the image is given by evaluating a valuation at the coordinates in $R[x,y]/(xy-f)$ (resp. $R[x]$).
\end{proof}

Let $F$ be any Hahn valued extension of $K$.

\begin{proposition}
Let $C$ be a geometrically irreducible $F$-curve. Then the space $C^\fH_\#$ is path connected.
\end{proposition}

\begin{proof} An arbitrary point $z$ of $C^\fH$ is represented by a valuation
	$$
	\val_z: \mathscr O_C(U)\lra \RR^{(k)}\sqcup \{\infty\}
	$$
for some Zariski open $U\subset C$. Let $E$ be the fraction field of $\mathscr O_C(U)\big/\val_z^{-1}(\infty)$, equipped with the valuation $\val_z$. By base changing to $E$, we obtain a continuous map $C^{\fH}_{E}:=(C\times_F E)^\fH\to C^\fH$ and an $E$-valued point $z'$ in $C_{E}$ mapping to $z$. Since $E$ is a valued field extension of $F$, the $F$-rational points of $C$ include into $C^{\fH}_{\!E}$. Thus it suffices to connect the $E$-rational points in $C^{\fH}_{\!E}$, for an arbitrary $\RR^{(k)}$-valued field extension $E$ of $F$.

Now consider two $E$-valued points $z$ and $w$ in $C_{E}$. As the normalization map is surjective, we reduce to the case that $C_{E}$ is smooth. Choose a compactification $\hat C_{E}$ of $C_{E}$, and a collection $\mathscr P = \{p_i\}$ of distinct points of $\hat C_{E}$ such that $x,y\in \mathscr P$ and the marked curve $(\hat C,\mathscr P)$ is stable. By the valuative criterion for properness of $\Mbar_{g,n}$ (possibly after a finite base change) we may take the stable marked model $\mathscr C$ of $(\hat{C}_{E},\mathscr P)$ over the valuation ring $R_E$ of $E$. After blowup, we may assume that the components of the special fiber of $\mathscr C$ do not have self-intersection. Now Proposition~\ref{prop: skeleton-embeds} above furnishes a path connected subspace of $C^\fH_{\!E\#}$ that contains the points $z$ and $w$ and the result follows.
\end{proof}

\begin{proof}[Proof of Theorem~\ref{thm: xH-euclidean-path-connectivity}]
As in the proof of the proposition above, we may reduce the claim to a proof that if $E$ is any Hahn valued field extending $K$, then for any two $E$-rational points $z,w\in X_{E}$, there exists a path in $X^{\fH}_{\!E\#}$ connecting $z$ and $w$. Furthermore, we may assume that $E$ is algebraically closed. Then by Bertini's theorem, as stated in~\cite[p. 53]{AbVar}, the points $z$ and $w$ lie on an irreducible $E$-curve $C$ in $X$. We may apply the above proposition to find a path connecting $z$ and $w$. Projecting to $X^\fH_{\#}$, the result follows.
\end{proof}

\begin{proof}[Proof of Theorem~\ref{thm: tropical-theorem}]
By definition, $\trop_{\#}(X)$ comes with a surjective continuous map $\trop:X_\#^{\fH}\twoheadrightarrow\trop_{\#}(X)$. Thus path connectivity of $\trop_{\#}(X)$ follows immediately from Theorem \ref{thm: xH-euclidean-path-connectivity}.
\end{proof}

\subsection{Connections to Banerjee's tropicalizations}\label{sec: banerjee} Recall~\cite{HLF} that a higher local field $K$ has value group isomorphic to $\ZZ^{(k)}$, and that its algebraic closure $K^{\mathrm{alg}}$ has value group isomorphic $\mathbb{Q}^{(k)}$. Given a subvariety $X$ of a split, $d$-dimensional algebraic torus over $K$, Banerjee \cite{Ban13} defines the tropicalization of $X$ to be the subset of $\RR^{k\times d}$ obtained by taking Euclidean closure of the image of the coordinatewise valution map $X(K^{\mathrm{alg}})\to \RR^{k\times d}$. In the present section, we refer to the resulting topological space as {\em Banerjee's tropicalization}, and denote it $\trop_{\rm Ban}(X)$. 

	The spaces $\mathrm{trop}_\#(X)$ and $\trop_{\rm Ban}(X)$ do not coincide. The essential reason for this is the following. A {\em half-space} in $(\RR^{(k)})^d$ is any subset
$$
H = \big\{\bm r = (\underline r_1,\ldots, \underline r_d)\in (\RR^{(k)})^d: u_{1}\underline{r}_{1}+\cdots+u_{d}\underline{r}_{d}+\gamma\geq 0\big\},
$$
for $u=(u_{1},\dots,u_{d})\in \ZZ^d$ and $\gamma\in \RR^{(k)}$. A halfspace is closed in the product-order topology on $(\RR^{(k)})^d$, but is not closed in the Euclidean topology on $\RR^{k\times d}$. For instance, consider the halfspace
$$
H' = \big\{\underline r\in \RR^{(2)}: \underline r\geq (0,0)\big\}.
$$
It coincides with the subset of $\RR^{(2)}$ given by the right-half plane, minus the (open) vertical negative axis, and thus it is not closed in $\RR^{2}$ (see Figure \ref{two possible closures}).
\begin{figure}[h!]
    \definecolor{sepia}{rgb}{.35,0,0}
    \definecolor{navy}{rgb}{0,0,0.5}
    \ \ \ \scalebox{.9}{$
    \begin{xy}
    (0,0)*+{
        \begin{tikzpicture}
        \fill[blue!20] (-1.8,-1.8) -- (1.8,-1.8) -- (1.8,1.8) -- (-1.8,1.8) -- (-1.8,-1.8);
        \fill[black!25!blue!35] (0,-1.8) -- (1.8,-1.8) -- (1.8,1.8) -- (0,1.8) -- (0,-1.8);
        \draw[black!65!blue, very thick,->] (0,0) -- (0,1.8);
        \draw[black!65!blue, very thick, dashed] (0,-1.8) -- (0,0);
        \fill[black!65!blue] (0,0) circle (.11);
        \fill[white] (0,0) circle (.075);
        \end{tikzpicture}
    };
    (9,10)*+{\mbox{{\smaller\smaller {\color{black} $\big\{\underline{r}\!>\!(0,0)\big\}$}}}};
    (-5,0)*+{\mbox{{\smaller\smaller (0,0)}}};
    (-10,14)*+{\mbox{{\color{navy} $\mathbb{R}^{(2)}$}}};
    \end{xy}
    $}    
\caption{{\smaller The half-space $H'=\big\{\underline{r}\in\mathbb{R}^{(2)}:\underline{r}>(0,0)\big\}$ has a larger closure in  the Euclidean topology on $\mathbb{R}^{2}$ than it does in the order topology on $\mathbb{R}^{2}$.}}
\label{two possible closures}    
\end{figure}
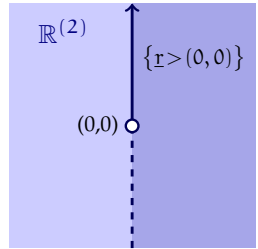

	For an explicit example of how the tropicalizations differ, let $X = V(x+y+1)$ in the $2$-dimensional torus. Then $\mathrm{trop}(X)$ consists of $3$ copies of the halfspace $H'$ glued at a their origins, whereas $\trop_{\rm Ban}(X)$ consists of three copies of the set $\big\{(r_{1},r_{2})\in\mathbb{R}^{2}:r_{1}\ge 0\big\}$ glued together at their origins.
	
	Despite the fact that they do not coincide, the two tropicalizations are related by Proposition \ref{proposition: Banerjee comparison} below. Together with Theorem~\ref{thm: tropical-theorem}, this resolves a question posed by Banerjee in~\cite[p. 2]{Ban13}.

\begin{proposition}\label{proposition: Banerjee comparison}
Let $K$ be a higher local field of rank $k$, and let $X$ be a subvariety of a $d$-dimensional split algebraic torus $T$ over $K$. Then $\trop_{\rm Ban}(X)$ is the closure of $\trop_\#(X)$ in the Euclidean topology. 
\end{proposition}

\begin{proof}
Choose a splitting of the torus $T$, inducing an isomorphism of $\Hom(M,\RR^{(k)})$ with $(\RR^{(k)})^d$. Let $P$ be a $\QQ^{(k)}$-rational polyhedron in $(\RR^{(k)})^d$. Identify $(\RR^{(k)})^d$ with the (unordered) abelian group $\RR^{kd}$, and consider $P$ as a Euclidean subset of $\RR^{kd}$. The set of points of $P$ with coordinates in $\QQ$ is dense in $P$. The value group of the algebraic closure of $K$ is $\QQ^{(k)}$, and thus $|\mathrm{trop}(X)|$ is a union of $\QQ^{(k)}$-rational polyhedra~\cite[Proposition 1.2, Remark 1.3]{NS11}. This implies that the set $|\mathrm{trop}(X)(\QQ)|$ of points of $\trop_\#(X)$ with coordinates in $\QQ$ is dense in $\mathrm{trop}_\#(X)$. It follows from~\cite[Proposition 1.1]{NS11} that $|\mathrm{trop}(X)(\QQ)|$ is precisely the image of $X(K^{\mathrm{alg}})$. The result follows.
\end{proof}

\begin{corollary}
If $X$ is a geometrically connected subvariety of a split algebraic $K$-torus, then $\trop_{\mathrm{Ban}}(X)$ is path connected. 
\end{corollary}

\subsection{Definability}\label{sec: definability} We require certain rudimentary notions from model theory. We give a brief and self-contained treatment, sufficient for our purposes. For a more detailed introduction and an overview of the work of Hrushovski and Loeser~\cite{HL10}, see~\cite{Duc2}. For background on $o$-minimal and definable structures, see~\cite[Chapter 3]{vdD98}.

	The composition $\rho\circ\nu:K^\times\to \RR^{(k)}$ is a valuation on $K$. Let $K^{\textrm{alg}}$ denote an algebraically closed field extending $K$, together with a valuation $\nu_{K^{\textrm{alg}}}:(K^{\textrm{alg}})^\times\to \RR^{(k)}$ extending $\rho\circ\nu$. Such an extension always exists since $\RR^{(k)}$ is divisible~\cite[Theorem 3.2.4]{EP05}. Denote by $\mathbf K$ the pair\footnote{As in~\cite[Secion 14]{HL10}, we view $\mathbf K$ as a substructure of the theory of algebraically closed valued fields, in valued field and value group sorts.} $(K^{\textrm{alg}},\RR^{(k)})$.

	Let $\mathbf{ACVF}$ be the category in which an object is any algebraically closed field $F$ extending $K^{\textrm{alg}}$ and equipped with a valuation $\nu_F:F^\times\to\RR^{(k)}$ extending the valuation on $K^{\mathrm{alg}}$. Morphisms are given by isometric embeddings over $K$.

	Let $X$ and $Y$ be finite-type $K$-schemes. A subfunctor $D$ of $X$ is called \textit{$\mathbf{K}$-definable} if it can be defined Zariski locally in $X$ by a (finite) boolean combination of inequalities of the form $\nu(f)\bowtie \lambda+\nu(g)$ in $\RR^{(k)}$, where $f$ and $g$ are regular functions, $\lambda\in \RR^{(k)}\sqcup \{\infty\}$, and $\bowtie\ \in \{\leq, \geq,<,>\}$. 
	
	A subset of $X(F)$ is said to be $\mathbf K$-definable if it can be defined Zariski locally in $X$ by a (finite) boolean combination of inequalities of the form $\nu(f)\bowtie \lambda+\nu(g)$. When $k=1$, this reduces to the familiar notion of a semi-algebraic subset of $X(F)$. Similarly, one defines $\mathbf K$-definable maps $X(F)\to Y(F)$ by requiring the graph to be a definable subset of the product.

\begin{definition}\label{definition: definable functors}
A functor $D:\mathbf{ACVF}\to \mathbf{Sets}$ is {\em $\mathbf{K}$-definable} if it is isomorphic to a quotient $E/R$ of a $\mathbf{K}$-definable subfunctor $E\subset X$ of the functor of points $X$ of some $K$-scheme by an equivalence relation $R\subset E\times E$ that is itself a $\mathbf{K}$-definable subfunctor of $X\times X$.
	
	A natural transformation $\varphi: D\to E$ between $\bold{K}$-definable functors is said to be {\em $\mathbf{K}$-definable} if the graph of $\varphi$ is a $\mathbf{K}$-definable functor. 
\end{definition}

If $\Delta$ is a $\mathbf K$-definable functor, for every $F\in \mathbf{ACVF}$, a subset of $\Delta(F)$ is a {\em $\mathbf K$-definable subset} if it can be written as $D(F)$ for some $\mathbf K$-definable subfunctor of $\Delta$. 

\subsection{Model theoretic connectivity} In this subsection, the choice of $\mathbf K$ will be implicit in the discussion, and we will use the term \textit{definable} in place of \textit{$\mathbf K$-definable}. 

There is a natural pairing
\begin{eqnarray*}
\langle-,-\rangle\ :\ \big(\RR^{(k)}\sqcup \{\infty\}\big)^d\times \ZZ^d&\lra& \RR^{(k)}\sqcup \{\infty\}\\
\big(\ \!(\gamma_1,\ldots, \gamma_d),\ \!(u_1,\ldots, u_d)\ \!\big)&\longmapsto& \sum u_i\gamma_i.
\end{eqnarray*}

\noindent
A \textit{rational halfspace} in $\big(\RR^{(k)}\sqcup\{\infty\}\big)^d$ is any set of the form
	$
	H_{\underline\gamma,\delta}\ :=\ \big\{\underline\gamma\in \big(\RR^{(k)}\sqcup\{\infty\}\big)^d:\langle \underline\gamma,u\rangle\geq \delta\big\},
	$
for fixed slope $u\in \ZZ^d$ and affine constraint $\delta\in \RR^{(k)}\sqcup \{\infty\}$. Its {\em boundary} is $ \big\{\underline\gamma\in \big(\RR^{(k)}\sqcup\{\infty\}\big)^d:\langle \underline\gamma,u\rangle=\delta\big\}$. A {\em rational polyhedron} $P$ in $\big(\RR^{(k)}\sqcup\{\infty\}\big)^d$ is any intersection of finitely many rational halfspaces
	\begin{equation}\label{equation: halfspace intersection}
	P\ =\ H_{\underline\gamma_{1},\delta_{1}}\cap\cdots\cap H_{\underline\gamma_{m},\delta_{m}}.
	\end{equation}
If $P$ is a rational polyhedron given by an intersection (\ref{equation: halfspace intersection}), then a {\em face} of $P$ is any intersection of $P$ with any number of the boundaries of the rational halfspaces in (\ref{equation: halfspace intersection}). A {\em rational polyhedral complex} in $\big(\RR^{(k)}\sqcup\{\infty\}\big)^d$ is a finite a collection $\{P_{j}\}_{j\in J}$ of polyhedra in $\big(\RR^{(k)}\sqcup\{\infty\}\big)^d$ such that every face of every $P_{j}$ also lies in the collection, and such that the intersection of any two polyhedra in the collection again lies in the collection.
	
	The following proposition follows from the discussion in~\cite[Section 1.4]{Duc2}.

\begin{proposition}
A subset of $\big(\RR^{(k)}\sqcup\{\infty\}\big)^d$ is definable if and only if it is isomorphic to a finite union of rational polyhedra. In particular, the set underlying any rational polyhedral complex in $\big(\RR^{(k)}\sqcup\{\infty\}\big)^d$ is a definable subset of $\big(\RR^{(k)}\sqcup\{\infty\}\big)^d$.
\end{proposition}

The analogous statements hold upon replacing $\RR^{(k)}\sqcup \{\infty\}$ with $\RR^{(k)}$. 

\begin{theorem}\label{thm: polyhedral-complex}
If $X$ is a subvariety of a $d$-dimensional torus over $K$, then $\trop(X)$ carries the structure of a rational polyhedral complex. In particular, $\trop(X)$ is a definable subset of $(\RR^{(k)})^d\cong \Hom_{\ZZ}(M, \RR^{(k)})$. 
\end{theorem}

\begin{proof}
This is proved for valued fields of arbitrary rank in~\cite[Proposition 1.2, Remark 1.3]{NS11} by choosing an embedding of the torus into $\PP^{d}$ and then using the Gr\"obner complex to give $\trop(X)$ a polyhedral complex structure.
\end{proof}

	We now give a more precise restatement of the connectivity part of Theorem~\ref{thm: tropical-theorem'}.
	
	A \textit{generalized interval} is any definable space obtained as follows: Given an interval $\big[\gamma_{a},\ \!\gamma_{b}\big]\subset \RR^{(k)}\sqcup \{\infty\}$, we may consider either $\big[\gamma_{a},\ \!\gamma_{b}\big]$ with its induced lexicographic order, or with its opposite lexicographic order. The choice of order is called an \textit{orientation} of the interval. A generalized interval is any definable space obtained from a collection of oriented intervals		
	$$
		\big[\gamma_{1a},\ \!\gamma_{1b}\big],\ \big[\gamma_{2a},\ \!\gamma_{2b}\big],\ \dots,\ \big[\gamma_{na},\ \!\gamma_{nb}\big]\ \ \subset\ \ \RR^{(k)}\sqcup \{\infty\}
		$$
by identifying the consecutive endpoints, respecting the orientations. That is, by identifying the largest endpoint of $[\gamma_{ma},\gamma_{mb}]$ with the smallest endpoint of $[\gamma_{(m+1)a},\gamma_{(m+1)b}]$.

\begin{theorem}\label{thm: definable-connectivity}
{\bf (Restatement of Theorem~\ref{thm: tropical-theorem'}).}
Let $X$ be a geometrically connected subvariety of a split algebraic torus over $K$. Given two points $x,y\in \trop(X)$, there exists a generalized interval $I$, together with a definable, continuous morphism $P:I\to \trop(X)$ whose endpoints map to $x$ and $y$.
\end{theorem}

We begin the proof with an analogue of Construction~\ref{valued graphs} for the extended order topology.

\begin{construction}\label{valued graphs'}
Let $G$ be a finite graph with edges $e_1,\ldots, e_r$. Fix a tuple
	$$
	\underline \ell\ =\ \big(\ell(e_1),\ldots, \ell(e_r)\big)\ \in\ \big(\RR^{(k)}_{\geq 0}\sqcup \{\infty\}\big)^r.
	$$
Denote by $|e_i|$ the interval $\big[0,\ell(e_i)\big]\subset \RR^{(k)}\sqcup \{\infty\}$, considered as a subspace in the extended order topology. Denote by $|G(\underline \ell)|$ the topological space obtained by gluing endpoints of $|e_i|$ and $|e_j|$ when the corresponding endpoints $e_i$ and $e_j$ are identified.
\end{construction}

	Let $\mathscr C$ be a proper, regular, marked semistable $R$-model with generic fiber $C$. Let $G$ be the marked dual graph of associated to $\mathscr{C}$ as in Remark \ref{remark: length assignment}, let $\ell(e_{i})\in\RR^{(k)}$ denote the length assigned to each edge $e_{i}$ in $G$, and let $|G(\underline{\ell})|$ the resulting topological space described in Construction \ref{valued graphs'}.

\begin{proposition}
There is a continuous embedding $|G(\underline \ell)|\hookrightarrow C^\fH$. Moreover, $e$ is an edge corresponding to a marked point $p$ of the generic fiber $C$, the infinite point of $|e|$ is mapped to the image of $p$ under the inclusion $C(K)\hookrightarrow C^\fH$. 
\end{proposition}

\begin{proof}
The proof is identical to that of Proposition~\ref{prop: skeleton-embeds}, replacing the Euclidean topology with the order topology throughout.
\end{proof}

The rank-$1$ analogue of the following proposition was proved in~\cite[Theorem 6.22]{BPR}. 

\begin{proposition}[\textbf{Faithful tropicalization for curves}]\label{prop: faithful-trop}
Let $\mathscr C$ be a marked model as above. There exists a rational map $\varphi: C\dashrightarrow \mathbb G_{\mathrm{m}}^n$ such that the tropicalization of $C$ with respect to $\varphi$ is injective on the subspace $|G(\underline \ell)|\subset C^\fH$.
\end{proposition}

\begin{proof}
The result follows from arguments similar to those given for the rank-$1$ case in the proof of \cite[Theorem 9.5]{GRW}. We explain how this proof may be adapted to the present context.
	
	If $\mathscr U$ is an affine open in $\mathscr C$, define $U:=\mathscr{U}_{K}$, and let $U^{\fH\ge0}$ denote the subspace of $U^\fH$ consisting of those points represented by morphisms $\spec_{\ \!}L\to U$ that extend to a morphism $\spec_{\ \!}R_{L}\to \mathscr U$, where $R_{L}$ denotes the valuation ring of $L$. There is a reduction map
	\begin{equation}\label{equation: local reduction maps}
	\red_{\mathscr{U}}:U^{\fH\ge0}\lra\mathscr{U}_{0}
	\end{equation}
given by sending a point $\spec_{\ \!}R_{L}\to \mathscr U$ to the image of the closed point of $\spec_{\ \!}R_{L}$. Since the model $\mathscr{C}$ is proper, $C^\fH$ is covered by the sets $U^{\fH\ge0}$, and one can check, as in the rank-$1$ case~\cite[Section 2]{MN12}, that the local reduction maps (\ref{equation: local reduction maps}) fit together to give a well defined reduction map
$$
\mathrm{red}_{\mathscr C}: C^\fH\lra\mathscr C_0.
$$

	Let $\mathfrak Z$ be the set consisting of the marks on $\mathscr{C}$, the generic points of the irreducible components of $\mathscr{C}_{0}$, and the nodes of $\mathscr{C}_{0}$. Projectivity and quasicompactness of $\mathscr C$ implies that for each $\zeta\in \mathfrak Z$, there exists a finite cover of $\mathscr{C}$ by affine opens $\mathscr U_{\zeta,j}$ containing $\zeta$. On each open $\mathscr U_{\zeta,j}$ there are finitely many regular functions $f_{\zeta, jk}$ whose reduction to the special fiber have zero set $\overline \zeta\cap \mathscr U_{\zeta,j}$. For every $p\in U^{\fH\ge0}_{\!\zeta,j}$, we have that $\val_p(f_{\zeta,jk}) = 0$ if the reduction of $p$ is not contained in the closure of $\zeta$, and $\val_p(f_{\zeta,jk}) = 0$ is strictly positive otherwise. 

	Near each node $q_i$ of $\mathscr{C}_{0}$, we have rational functions $x_i$ and $y_i$ cutting out the components meeting at $q_i$. Similarly, for each marked section $s_r$, we obtain a rational function $z_r$. The collection of functions $\{x_i,y_i,z_r,f_{\zeta,jk}\}$ determine a rational map $C\dashrightarrow \mathbb G_{\mathrm{m}}^n$. If $p,p'\in |G(\underline \ell)|\subset C^\fH$ are points whose reductions $\red_{\mathscr{C}}(p)$ and $\red_{\mathscr{C}}(p')$ lie in adjacent strata of $\mathscr{C}_{0}$, then some pair of functions $x_{i}$ and $y_{i}$ or $x_{i}$ and $z_{r}$ separate $p$ and $p'$. If the reductions $\red_{\mathscr{C}}(p)$ and $\red_{\mathscr{C}}(p')$ lie in non-adjacent strata of $\mathscr{C}_{0}$, there is a rational function $f_{\zeta,jk}$ in our collection such that $\val_p(f_{\zeta,jk})\neq \val_{p'}(f_{\zeta,jk})$. One checks as in the concluding paragraph of the proof of~\cite[Theorem 9.5]{GRW} that there exists a $\zeta\in \mathfrak Z$ and a function $f_{\zeta,jk}$ such that $\val_p(f_{\zeta,jk})$ is strictly positive, but $\val_{p'}(f_{\zeta,jk})$ is $0$.
\end{proof}

\begin{proof}[Proof of Theorem~\ref{thm: definable-connectivity}]
Let $X$ be a geometrically connected subvariety of a $d$-dimensional torus over $K$. Choose $z,w\in \trop(X)$. After base changing $X$ to a sufficiently large field $F$ over $K$, we can assume that there exist points $z',w'\in X_{F}(F)\subset X^\fH_{F}$ mapping to $z$ and $w$ under tropicalization. As before, Bertini's theorem produces an irreducible $F$-curve $C$ in $X_{F}$ connecting $z'$ and $w'$. Let $\hat C$ be an $F$-curve compactifying the normalization of $C$. Choose a semistable $R_F$-model for $\hat C$ marked at $z'$, $w'$. Denote by $|G|\subset C^\fH$ the graph constructed as in the proposition above. Since $C$ is connected, $G$ is also connected, and one may find a generalized interval $I\subset |G|\subset C^\fH$. Composing with the tropicalization map, we obtain a map $P:I\to \trop(X)$ connecting $z$ and $w$. Continuity is clear, and so it remains only to show the definability of this path. The curve $C$ is embedded in the torus $\mathbb G_{\mathrm{m}}^d$. Using Proposition~\ref{prop: faithful-trop} above, one may enlarge the embedding set of functions for $C$ to obtain a rational morphism $\psi: C\dashrightarrow\mathbb G_{\mathrm{m}}^{d+m}$, such that the tropicalization of $C$ with respect $\psi$ is injective on $|G|$ and hence on $I$. The path $P$ factors as
	$$
	\begin{xy}
	(0,0)*+{I}="1";
	(23,0)*+{(\RR^{(k)})^{d+m}}="2";
	(50,0)*+{(\RR^{(k)})^{d}}="3";
	{\ar "1"; "2"};
	{\ar "2"; "3"};
	{\ar@/^15pt/^{P} "1"; "3"};
	\end{xy}.
	$$
The first map is an embedding of a $1$-dimensional polyhedral complex and is consequently definable. The second map is projection onto the first $d$ factors, hence also definable. Thus $P$ is definable and continuous.
\end{proof}

The arguments above imply the following connectivity property of $X^\fH$.

\begin{theorem}\label{thm: xH-galois-connectivity}
Let $X$ be a geometrically connected $K$-variety. Given any two points $z,w\in X^\fH$, there exists a Hahn valued field $F$ extending $K$, points $\widetilde z,\widetilde w\in (X\times_K F)^\fH$ mapping to $z$ and $w$, and a continuous definable path $P: I\to (X\times_K F)^\fH$ connecting $\widetilde z$ and $\widetilde w$.
\end{theorem}

\section{Comparison results}\label{sec: comparison}

We briefly explain how the Hahn analytification relates to the Huber analytification and the Hrushovski-Loeser spaces.

\subsection{Comparison results I: Huber adic space} Let $K$ be a field, complete with respect to a nontrivial rank-$1$ valuation. For any affine $K$-scheme $X=\mathrm{Spec}_{\ \!}K[X]$, the {\em adic space} $X^{\mathrm{ad}}$ associated to $X$ is the set
	\begin{equation}\label{Huber analytification def}
	\begin{aligned}
	X^{\mathrm{ad}}
	\ \ \ =\ \ \ 
	\left.
	\left\{
		\begin{array}{c}
		\mbox{valuations $K[X]\lra\Upsilon\sqcup\{\infty\}$ restricting}\\
		\mbox{to $K\rightarrow v_{K}(K^{\times})\sqcup\{\infty\}\hookrightarrow\Upsilon\sqcup\{\infty\}$}
		\end{array}
	\right\}
	\right/
	{}_{\mbox{{\larger\larger $\sim$}}}\ \ ,
	\end{aligned}
	\end{equation}
where ``$\sim$" is the equivalence relation generated by the requirement that two valuations $K[X]\lra\Upsilon_{1}\sqcup\{\infty\}$ and $K[X]\lra\Upsilon_{2}\sqcup\{\infty\}$ be equivalent if there exists an inclusion $\Upsilon_{1}\hookrightarrow\Upsilon_{2}$ of totally ordered abelian groups such that the diagram
	$$
	\begin{aligned}
	\begin{xy}
	(0,0)*+{\ \ \Upsilon_{2}\sqcup\{\infty\}}="1";
	(0,-20)*+{\ \ \overset{}{\Upsilon_{1}\sqcup\{\infty\}}\!\!}="2";
	(-18,-10)*+{K[X]}="3";
	{\ar@{_{(}->} "2"; "1"};
	{\ar "3"; "1"};
	{\ar "3"; "2"};
	\end{xy}
	\end{aligned}
	$$
commutes. The space $X^{\mathrm{ad}}$ is equipped with the topology generated by sets of the form
	$$
	U\big(\tfrac{f}{g}\big)
	\ \ :=\ \ 
	\big\{
	x\in X^{\mathrm{ad}}\ :\ \mathrm{val}_{x}(g)<\mathrm{val}_{x}(f)
	\big\},
	$$
where $f$ and $g$ are regular functions on $X$. 
	
	If we fix an embedding of totally ordered abelian groups $\Gamma \hookrightarrow\mathbb{R}^{(k)}$, then every valuation $\mathrm{val}_{x}:K[X]\lra\mathbb{R}^{(k)}$ that describes a point $x$ of the resulting Hahn analytification $X^{\fH}$ satisfies $\mathrm{val}_{x}\big|_{K}=v_{K}$. Thus we obtain a map
	$$
	\eta\ :\ X^{\fH}\ \lra\ X^{\mathrm{ad}}.
	$$

\begin{theorem}\label{thm: huber-comparison}
If we choose our Hahn embedding $\Gamma\hookrightarrow\mathbb{R}^{(k)}$ to be the inclusion $r\mapsto(r,0,\dots,0)$ into the $1^{\mathrm{st}}$ factor of $\mathbb{R}^{(k)}$, and if
	$$
	k
	\ \ \ge\ \ 
	1+\mathrm{dim}_{\mathrm{Krull}\ \!}X,
	$$
then the map $\eta:X^{\fH}\lra X^{\mathrm{ad}}$ is surjective.
\end{theorem}

\begin{proof}
We may assume that $X$ is affine with coordinate ring $K[X]$. Consider a point $x\in X^{\rm ad}$ represented by the valuation $\nu: K[X]\to \Upsilon\sqcup\{\infty\}$, and note that this valuation induces an inclusion $\Gamma\hookrightarrow\Upsilon$.
	
	The chain of convex subgroups of $\Upsilon$ is in bijection with a chain of prime ideals of $K[X]$ (see ~\cite[Section 2.3]{EP05}). Since $X$ is finite dimensional, this chain of convex subgroups is finite. Thus we may assume that $\Upsilon$ has finite rank bounded by $\dim(X)+1$. Since every convex subgroup is a union of archimedean equivalence classes of $\Upsilon$, we deduce that $\Upsilon$ admits an order-preserving embedding $\Upsilon\hookrightarrow \RR^{(k)}$. It is straightforward to check that this embedding can be chosen to be compatible with the embedding $\Gamma\hookrightarrow\mathbb{R}^{(k)}$. In this way, we obtain a valuation
$$
K[X]\lra \Upsilon\sqcup\{\infty\}\mono \RR^{(k)}\sqcup\{\infty\},
$$
and thus a point $x_h$ of $X^\fH$. It is clear that $x_h\mapsto x$. 
\end{proof}

\begin{remark}
The map $\eta:X^{\fH}\lra X^{\mathrm{ad}}$ will not be continuous in general. For instance, let $\Gamma=\mathbb{R}$ and $k=1$. Then $X^{\fH}=X^{\mathrm{an}}$, and $\eta$ becomes the section
	\begin{equation}\label{section}
	X^{\mathrm{an}}\lra X^{\mathrm{ad}}
	\end{equation}
of the maximal Hausdorff quotient map $X^{\mathrm{ad}}\lra\!\!\!\!\rightarrow X^{\mathrm{an}}$, sending a higher rank valuation $\val:\mathscr O(U)\to \Upsilon\sqcup\{\infty\}$ to the valuation $\mathscr O(U)\to \Upsilon \sqcup\{\infty\}\to \big(\Upsilon/\Upsilon_{\!1}\big)\sqcup\{\infty\}$,
where $\Upsilon_{\!1}$ is the largest proper convex subgroup of $\Upsilon$. This section ({\ref{section}}) is not continuous in general~\cite[Proposition 8.3.1, Lemma 8.1.8]{Hub}.
\end{remark}

\subsection{Comparison results II: Stable completion} In~\cite{HL10}, Hrushovski and Loeser associate to any variety $X$ over an algebraically closed valued field $K$ a space $\widehat{X}$ called the \textit{stable completion of $X$}. The relationship between $\widehat X$ and $X^\fH$ is similar to the relationship between $\widehat X$ and the Berkovich analytification $X^{\mathrm{an}}$ when $K$ has a nontrivial rank-$1$ valuation, as we now explain. The surveys by Ducros~\cite{Duc1,Duc2} provide an essentially self-contained introduction to the model theoretic background to this section. The reader may safely skip this section, as the rest of the paper does not logically depend upon it.

\subsection*{The stable completion functor} Let $K$ be an algebraically closed valued field with value group $\Gamma$. Let $\mathbf M$ denote the category of algebraically closed valued extensions of $K$, where morphisms are taken to be isometric $K$-embeddings. For $F\in \mathbf M$, let $\mathbf M_F$ be the category of valued extensions of $F$ that belong to $\bf M$. Observe that this category $\bf M$ differs in a small but important way from $\bf{ACVF}$, defined previously. In the latter, one additionally keeps track of the particular embedding of the value group into $\RR^{(k)}$. 

Let $X = \spec(A)$ be an affine $F$-scheme. A \textit{type} $t$ on $X$ is an element of the valuative spectrum of $X$. That is, $t$ is the data of a scheme theoretic point $x$ of $X$, together with a valuation $\nu_x$ on the residue field of $X$ at $x$. in particular, a type $t$ thus gives rise to a valuation $\varphi_t$ on $A$.

A type $t$ is said to be \textit{$F$-definable} if and only if for every finite-dimensional $F$-subspace $E$ of $A$, the following subsets are $F$-definable in the sense of Definition \ref{definition: definable functors}:
\begin{itemize}
\item The set of elements $e\in E$ such that $\varphi_t(e) = \infty$;\vskip .2cm
\item The set of elements $e\in E$ such that $\varphi_t(e)\geq 0$.
\end{itemize}

A type $t$ is said to be \textit{orthogonal to $\Gamma$} if and only if it is $F$-definable and $\varphi_t$ takes values in $\nu(F)$. 

\begin{definition}
Let $X$ be an affine $K$-scheme of finite type and let $F\in \mathbf M$. The \textit{stable completion of $X$ at $F$} is the set $\widehat X(F)$ of types on $X\times_K F$ that are orthogonal to $\Gamma$. 
\end{definition}

	As with the Berkovich and Hahn analytifications, the set $\widehat X(F)$ is given the weak topology for the evaluation functions $\ev_f: \widehat X(F)\to \Gamma\sqcup\{\infty\}$ for all $f\in A$. The construction extends to arbitrary finite-type $K$-schemes in the natural way.

\subsection*{The comparison map} Suppose $K$ is complete with respect to a rank-$1$ valuation and that $F\in \mathbf M$ also has rank-$1$. An element of $\widehat X(F)$ can be interpreted as a valuation with values in $\nu(F)\subset \RR\sqcup \{\infty\}$. This gives rise to a continuous map
	\begin{equation}\label{equation: comparison to stable completion}
	\pi_F: \widehat X(F)\lra X^{\mathrm{an}}
	\end{equation}

Hrushovski and Loeser prove the following result about the comparison map $\pi_F$.

\begin{proposition}[{\cite[Lemma 14.1.1]{HL10}}]\label{HL's comparison}
If the valuation $F^{\times}\to\RR$ is surjective, and if $F$ is maximally complete, then the map (\ref{equation: comparison to stable completion}) is a proper surjection. 
\end{proposition}

	The map (\ref{equation: comparison to stable completion}) plays a crucial role in \cite[Section 14]{HL10} for deducing tameness results about the Berkovich space from results about the spaces of types. The following extension of Proposition \ref{HL's comparison} gives further evidence that the Hahn analytification $X^\fH$ is an analogue of the Berkovich analytification in the higher rank setting.

	Let $K$ be an arbitrary algebraically closed valued field with Hahn valuation
$$
K^\times\xrightarrow{\ \ \nu\ \ }\Gamma\xrightarrow{\ \ \rho\ \ } \RR^{(k)}.
$$
Assume that $\rho$ is an embedding of ordered abelian groups. Note that this is not a serious restriction since we may always replace $\Gamma$ with the image of $\Gamma$ under $\rho$. 

\begin{proposition}\label{Hahn stable completion comparison}
If $F$ is a maximally complete, algebraically closed field extending $K$, with surjective valuation $F^{\times}\twoheadrightarrow\RR^{(k)}$ extending the valuation on $K$, then there is a natural continuous surjection
	$$
	\pi_F:\widehat X(F)\lra X^\fH.
	$$
\end{proposition}

\begin{proof}
Consider a point $x\in X^\fH$ associated to a valuation $\val_x:A\to \mathbb{R}^{(k)}_{\infty}$. The prime ideal $ \val_x^{-1}(\infty)$ gives rise to a scheme theoretic point of $X$, and hence a map $\spec(L)\to X$. The valuation $\val_x$ gives rise to a valuation on the field $L$. We may assume that $L$ is algebraically closed. Let $L^{\max}$ denote the field (unique by Kaplansky's theorem \cite{Kap42}) having value group $\RR^{(k)}$ and residue field equal to the residue field of $L$. We may represent the point $x$ by a map $\spec(L^{\max})\to X$. The field $F$ includes into $L^{\max}$. Thus we obtain a type $t$ on $\spec(A\otimes_K F)$. Since $F$ is maximally complete, and since the value group $\RR^{(k)}$ has no Archimedean extensions, we may apply the result of Haskell, Hrushovski, and Macpherson~\cite[Theorem~12.18]{HHM} to conclude that the type is orthogonal\footnote{See also the statement and remarks following~\cite[Theorem 2.9.2]{HL10}}. It is clear that this type $t$ maps to the point $x$ under $\pi_{F}$, and surjectivity follows.  
\end{proof}

\begin{remark}
Note that fields $F$ extending a Hahn valued field $K$ as in Proposition \ref{Hahn stable completion comparison} above always exist. Indeed, the group ring $$
	K\big[ t^{\mathbb{R}^{(k)}}\big]
	\ \ :=\ \ 
	\left\{
		\begin{array}{c}
		\mbox{all sums $f(t)=\sum_{\underline{r}\in\mathbb{R}^{(k)}}\!a_{\underline{r}}t^{\underline{r}}$}\\
		\mbox{with finite support $\mathrm{supp}_{\ \!}f(t)$}\\
		\end{array}
	\right\}
	$$
comes with a surjective map
	$$
	\nu_{\mathrm{mon}}\ :\ \ 
	K\big[ t^{\mathbb{R}^{(k)}}\big]
	\ \ \lra\!\!\!\!\rightarrow\ \ 
	\mathbb{R}^{(k)}\sqcup\{\infty\},
	$$
given by $\nu_{\mathrm{mon}}(0):=\infty$ and
	$$
	\nu_{\mathrm{mon}}\big(f(t)\big)\ :=\underset{\underline{r}\in\mathrm{supp}_{\ \!}f(t)}{\mathrm{inf}}\big(\nu_{K}(a_{\underline{r}})+\underline{r}\big)
	$$
for nonzero $f(t)=\sum_{\underline{r}\in\mathbb{R}^{(k)}}\!a_{\underline{r}}t^{\underline{r}}$. One checks, as in ~\cite[Proposition 2.1.2]{Ked}, that this defines a valuation. Passing to a maximally complete algebraic closure of this field yields the desired extension.
\end{remark}

\section{Extended tropicalization and an inverse limit theorem}\label{ref: tropicalization}

\subsection{Tropicalization in toric varieties}\label{extended tropicalization} Let $K$ be a field with valuation $\nu:K\to\RR^{(k)}\sqcup\{\infty\}$.
	
	The Hahn tropicalization of $K$-tori and their subvarieties extends naturally to subvarieties of toric varieties, generalizing the construction of Kajiwara and Payne~\cite{Kaj08, Pay09}. Let $M$ be a lattice and $N = \Hom_{\ZZ}(M,\ZZ)$ the dual lattice. Denote by $U(\sigma) = \spec\big(K[S_\sigma]\big)$ the affine toric variety associated to a cone $\sigma$ in $N_\RR$. Consider the set $N(\sigma)$ of semigroup homomorphisms from the commutative semigroup $S_{\sigma}$ into the commutative semigroup underlying $\RR^{(k)}\sqcup\{\infty\}$, i.e.,
$$
N(\sigma)\ :=\ \Hom_{\bold{SGrp}}\big(S_\sigma,\RR^{(k)}\sqcup\{\infty\}\big). 
$$
Give $N(\sigma)$ the subspace topology for its inclusion into $(\RR^{(k)}\sqcup\{\infty\}\big)^{S_\sigma}$, where $\RR^{(k)}\sqcup\{\infty\}$ is given the extended order topology. Let $x = \val_x:K[S_{\sigma}]\to\RR^{(k)}\sqcup\{\infty\}$ be a point of $U(\sigma)^\fH$. Restricting to $S_\sigma$, we obtain a point $\mathrm{trop}(X)$ of $N(\sigma)$. This furnishes a well defined, continuous tropicalization map
$$
\trop: U(\sigma)^\fH\lra N(\sigma).
$$

Let $\Delta$ be a fan in $N_\RR$ with associated toric variety $Y(\Delta)$. If $\tau$ is a face of some cone $\sigma$ in $\Delta$, then one may restrict functions from $U(\sigma)$ to $U(\tau)$ to obtain a map $S_\tau\hookrightarrow S_\sigma$ and consequently an embedding of $N(\tau)$ into $N(\sigma)$. Gluing along these inclusions, we obtain a topological space $N(\Delta)$. The tropicalizations on affine invariant patches can be glued to form a continuous map
	\begin{equation}\label{equation: extended trop}
	\trop: Y(\Delta)^\fH\lra N(\Delta).
	\end{equation}
	
The map above may be seen as a projection of $Y(\Delta)^\fH$ onto a closed subspace.	

\begin{proposition}
The tropicalization map (\ref{equation: extended trop}) admits a continuous section $s: N(\Delta)\to Y(\Delta)^\fH$.
\end{proposition}
\begin{proof}
It suffices to consider the affine case $Y(\Delta) = \spec_{\ \!} K[S_\sigma]$ for a single cone $\sigma$ in $\Delta$. Given a point $\omega \in N(\sigma)$, one constructs a monomial valuation with weight $\omega$, as in~\cite[Proposition 2.9]{Thu07}, and these monomial valuations provide the desired section. 
\end{proof}

Given a torus-equivariant map of toric varieties $\varphi: Y(\Delta)\to Y(\Delta')$, one may pullback characters and, following the construction above, obtain a map $\varphi^{\trop}: N(\Delta)\to N(\Delta')$. In particular, tropicalization is covariantly functorial.
	
A closed embedding $\iota: X\hookrightarrow Y(\Delta)$ induces a continuous embedding $X^\fH\hookrightarrow Y(\Delta)^\fH$. Composing with (\ref{equation: extended trop}) yields a subspace $\trop(X,\iota)$ of $N(\Delta)$. If $U$ is an open subscheme of $X$, then we let $\trop(U,\iota)$ denote the image of $U^{\fH}$ under $\trop:X^{\fH}\to\trop(X,\iota)$.

\subsection{Inverse limit theorem and prodefinability of $\pmb{X^\fH}$}
For this section, we assume that $K$ is algebraically closed. Let $X$ be a $K$-variety admitting a closed embedding into at least one toric variety. Covariant functoriality for torus-equivariant morphisms between toric varieties gives us a nonempty, inverse system $\mathcal{S}$ of tropicalizations of $X$ under closed embeddings into toric varieties. This, in turn, gives rise to a continuous map
	\begin{equation}\label{equation: tau}
	\tau: X^\fH\lra \varprojlim \trop(X,\iota). 
	\end{equation}

\begin{theorem}\label{thm: inverse-limit}
If $X$ is a $K$-variety admitting a closed embedding into at least one toric variety, then the map (\ref{equation: tau}) is a homeomorphism.
\end{theorem}
\begin{proof}
The proof is closely modeled after the proofs of \cite[Theorem 4.2]{Pay09} and \cite[Theorem 1.2]{FGP}. Fix, at the outset, a single closed embedding $\iota_{0}:X\hookrightarrow Y(\Delta_{0})$ into a toric variety. Let $\mathcal{S}_{0}$ denote the subcategory of $\mathcal{S}$ consisting of all those closed embeddings that factor through $\iota_{0}$. Then it suffices to prove that the map
	\begin{equation}\label{equation: tau_{0}}
	\tau_{0}: X^\fH\lra \varprojlim_{\iota\in\mathcal{S}_{0}} \trop(X,\iota)
	\end{equation}
is a homeomorphism.

	The arguments of \cite[Section 4]{FGP} make no use of the valuation on $K$. They require only that $K$ be algebraically closed. This means that there exists a finite open affine cover $\{U_{1},\dots,U_{r}\}$ of $X$ with the following property:
	\begin{itemize}
	\item[{ (${\star}$)}]
	For any $1\le i\le r$ and any nonzero regular function $f\in K[U_{j}]$, there exists a closed embedding $\iota_{f}:X\hookrightarrow Y(\Delta_{f})$ in $\mathcal{S}_{0}$ such that $U_{j}$ is the $\iota_{f}$-preimage of a torus-invariant open subset of $Y(\Delta_{f})$, and such that $f$ is the pullback of a monomial on $Y(\Delta_{f})$. 
	\end{itemize}
	
	We claim that for each open set $U_{j}$ in this cover, the map (\ref{equation: tau_{0}}) restricts to a homeomorphism of $U^{\fH}_{j}$ onto the preimage of $\trop(U_{j},\iota_{0})$ in $\varprojlim_{\iota\in\mathcal{S}_{0}} \trop(X,\iota)$.

	To verify injectivity, note that if $x\ne y\in U^{\fH}_{j}$, then there exists a regular function $f\in K[U_{j}]$ with $\mathrm{val}_{x}(f)\ne\mathrm{val}_{y}(f)$. Thus { (${\star}$)} provides us with a closed embedding $\iota\in\mathcal{S}_{0}$ such that the images of $x$ and $y$ in $\trop(X,\iota)$ are distinct.

	To verify surjectivity, fix a point inside the preimage of $\trop(U_{j},\iota_{0})$ in $\varprojlim_{\iota\in\mathcal{S}_{0}} \trop(X,\iota)$. Such a point is given by a compatible system $(v_{\iota})_{\iota\in\mathcal{S}_{0}}$ of points $v_{\iota}\in\trop(X,\iota)$, with $v_{\iota_{0}}\in\trop(U_{j},\iota_{0})$. For each $f\in K[U_{j}]$, use { (${\star}$)} to produce a closed embedding $\iota\in\mathcal{S}_{0}$ where $f$ becomes the pullback of a monomial $a\chi_{f}^{u}$, with $a\in K$. We claim that the assignment
	\begin{eqnarray*}
	\mathrm{val}_{x}\ :\ \ K[U_{j}]&\lra& \RR^{(k)}\sqcup\{\infty\}\\
	f&\longmapsto&\nu(a)+\nu_{\iota}(u)
	\end{eqnarray*}
is a well defined multiplicative valuation on $K[U_{j}]$. Well definedness of $\mathrm{val}_{x}$ follows from compatibility of the system $(v_{\iota})_{\iota\in\mathcal{S}_{0}}$ combined with the fact that $\mathcal{S}_{0}$ is closed under products. To verify multiplicativity and subadditivity, observe that the use of W\l odarczyk's algorithm \cite[Proof of Lemma~4.2]{Wlo} in the construction of the closed embedding in (${\star}$) gives us enough flexibility to construct closed embeddings $\iota\in\mathcal{S}_{0}$ in which any triple of functions, $f$, $g$, and $f+g$ in $K[U_{j}]$ become pullbacks of monomials.
	
	The definition of the topology on $U^{\fH}_{j}$ along with the existence of the closed embedding $\iota_{f}$ for each $f\in K[U_{j}]$ implies that the topology on $U^{\fH}_{j}$ coincides with the inverse limit topology on the preimage of $\trop(U_{j},\iota_{0})$ in $\varprojlim_{\iota\in\mathcal{S}_{0}} \trop(X,\iota)$. The local result at each $U_{j}$ in the open cover of $X$ implies the global result, i.e., that (\ref{equation: tau_{0}}) is a homeomoprhism. 
\end{proof}

The inverse limit statement of Theorem \ref{thm: inverse-limit} immediately implies the following model theoretic consequence. See~\cite[Section 3.1]{HL10} for the corresponding statement for the stable compleition.

\begin{theorem}\label{cor: prodefinability}
Let $X$ be a variety admitting at least one closed embedding into a toric variety. Then the space $X^\fH$ is a $\mathbf K$-prodefinable set in the sense of~\cite[Section 2.2]{HL10}.
\end{theorem}

\begin{remark}{\bf Scheme theoretic and universal tropicalization. } The constructions presented in this paper are compatible with recent work in which Giansiracusa and Giansiracusa develop a ``scheme theoretic'' framework for tropicalization \cite{GG13, GG14}. Given a subvariety $X$ of a toric variety over a Hahn valued field $K$, one may consider the associated tropical scheme $\mathcal T\textrm{rop}(X)$, as constructed in~\cite{GG13}. The homomorphism $\rho:\Gamma\to \RR^{(k)}$ allows one to consider the $\RR^{(k)}\sqcup\{\infty\}$-valued points of $\mathcal T\textrm{rop}(X)$. This underlying set of points coincides with the extended Hahn tropicalization above. Similarly, the $\RR^{(k)}\sqcup\{\infty\}$-valued points of the \textit{universal embedding} of $X$ constructed in~\cite{GG14} coincide with the underlying point set of the Hahn analytification. Together with Theorem~\ref{thm: huber-comparison}, this makes precise the observation~\cite[pg.3]{GG14} that the tropicalization of the universal embedding contains information about the Huber adic space.
\end{remark}

\bibliographystyle{siam}
\bibliography{HighRankDegenerations}

\end{document}